\documentclass[12pt]{article}
\usepackage[T1]{fontenc}
\usepackage[utf8]{inputenc}
\usepackage{amssymb,amsmath,amsfonts,amsthm,amscd,latexsym,verbatim,graphics,epsfig,indentfirst,xcolor}
\usepackage{geometry}
\def\Aut{\mathrm{Aut}}

\def\id{\mathop{\rm id}}

\def\Der{\mathop{\rm Der}}
\def\pr{\mathrm{pr}}
\def\ad{\mathrm{ad}}

\def\Spec{\mathop{\rm Spec}}
\def\PreLie{\mathop{\rm preLie}}
\def\PostLie{\mathop{\rm postLie}}
\def\PreAs{\mathop{\rm preAs}}
\def\PostAs{\mathop{\rm postAs}}

\newcommand{\tr}{\mathbin{\triangleright}}

\newtheorem{theorem}{Theorem}[section]
\newtheorem{proposition}[theorem]{Proposition}
\newtheorem{lemma}[theorem]{Lemma}
\newtheorem{corollary}[theorem]{Corollary}
\newtheorem{definition}[theorem]{Definition}

\newtheorem{example}[theorem]{Example}

\begin{document}

\begin{flushright}
\end{flushright}

\begin{center}
{\Large Failure of Ado-type theorems for preLie and postLie algebras}

\smallskip

Vsevolod Gubarev
\end{center}

\begin{abstract}
Over an arbitrary field of characteristic zero, we construct a two-dimensional preLie algebra and a two-dimensional postLie algebra that do not embed into the corresponding induced structures of any finite-dimensional preassociative or postassociative algebra, respectively.

\end{abstract}

\section{Introduction}

The classical Ado theorem states that every finite-dimensional Lie algebra over a field of characteristic zero has a faithful finite-dimensional representation, or, equivalently, embeds into the commutator Lie algebra of a~finite-dimensional associative algebra~\cite{Ado-0,Ado}.  

Analogues of the Ado theorem hold for Malcev and nilpotent Sabinin algebras~\cite{MostovoyPerezShestakov},
for Leibniz algebras~\cite{KolesnikovLeibniz},
and for torsion-free conformal Lie algebras with a splitting solvable radical~\cite{KolesnikovAdo}.

The main goal of the paper is to provide counterexamples to the analogues of the Ado theorem
for the pairs (preLie,preAs) and (postLie,postAs).

PreLie algebras appeared independently in the works of E.B.~Vinberg and M.~Gerstenhaber~\cite{Vinberg,Gerstenhaber}, see also the surveys~\cite{Burde,Manchon}.  
The two-operation algebras now called preassociative algebras were introduced by J.-L.~Loday~\cite{LodayDialgebras}.  
Every preassociative algebra $(A,\prec,\succ)$ carries a~natural preLie product $x\tr y=x\succ y-y\prec x$.  

The corresponding universal enveloping problem has a substantial history.
In the terminology of PBW-pairs introduced by A.A.~Mikhalev and I.P.~Shestakov
\cite{MikhalevShestakov}, one asks both for injectivity of the canonical map
and for~a PBW description of the universal enveloping algebra. 
For preLie algebras, the PBW problem and a special form of the enveloping-algebra problem were posed by P.~Kolesnikov~\cite{KolesnikovPreAs}; as recorded in~\cite{DotsenkoTamaroff}, J.-L.~Loday also asked V.~Dotsenko about the PBW problem.
A~positive solution of the embedding problem was obtained in~\cite{GubarevPreLie}, while the PBW property was proved by V.~Dotsenko and P.~Tamaroff~\cite{DotsenkoTamaroff}.  
The restricted positive-characteristic setting was considered by I.~Dokas~\cite{Dokas}.
A~unified construction of the universal enveloping preassociative algebra, together with another proof of
the PBW statement, was given in~\cite{GubarevPBW}.  
Related Hopf-algebraic results for preassociative and brace algebras go back to F.~Chapoton and M.~Ronco~\cite{Chapoton,Ronco}.

The notion of a postLie algebra was introduced by B.~Vallette~\cite{Vallette} and has since arisen in differential geometry, numerical integration, integrable systems, and the theory of affine actions; see~\cite{BaiGuoNiPost,BurdeDekimpeVercammen,EFLM}.
Their associative companion is the variety of postassociative algebras
introduced earlier by J.-L.~Loday and M.~Ronco~\cite{LodayRonco}.  
The embedding of postLie algebras into postassociative algebras was established in~\cite{GubarevPostLie}, and the universal enveloping construction and the PBW property were obtained in~\cite{DotsenkoPost,GubarevPBW}.  
The terms \emph{dendriform} and \emph{tridendriform} are also used for preassociative and postassociative algebras, respectively. We do not use this terminology below.

Rota--Baxter operators provide the main bridge between these structures.  
In weight zero, M.~Aguiar observed that an associative Rota--Baxter algebra carries a~natural preassociative structure~\cite{Aguiar}.  
The nonzero-weight postassociative analogue was obtained by K.~Ebrahimi-Fard~\cite{EbrahimiFard}.  
Universal enveloping Rota--Baxter algebras of preassociative algebras were studied in~\cite{EbrahimiFardGuo,ChenMo}; the embedding theorem for general pre- and postalgebras was proved in~\cite{GK}.  
Universal enveloping associative Rota--Baxter algebras of preassociative and postassociative algebras were constructed in~\cite{GubarevRBAs}.

We write down the natural Ado-type questions for the pairs $(\PreLie,\PreAs)$ and $(\PostLie,\PostAs)$:
\begin{enumerate}
\item Does every finite-dimensional preLie algebra embed into the induced
preLie algebra of a finite-dimensional preassociative algebra?
\item Does every finite-dimensional postLie algebra embed into the induced
postLie algebra of a finite-dimensional postassociative algebra?
\end{enumerate}

Problem~1.1 in~\cite{GubarevPBW} formulates the corresponding universal
embedding, PBW and enveloping-algebra problems for both pairs.  
These two questions impose the additional requirement that the target algebra itself be
finite-dimensional; they are therefore Ado-type refinements of that circle of problems.  
We show that, unlike the classical Lie--associative situation, both refinements have a
negative answer already in dimension two.

For the preLie case we use a one-generator preassociative argument and the
nilpotence of finite-dimensional Zinbiel algebras~\cite{Towers}.  
For the postLie case we pass to a~splitting finite-dimensional associative Rota--Baxter algebra and thereby obtain a~necessary condition for a~potential Ado-type theorem.
The resulting counterexample has an abelian initial Lie algebra and a~solvable induced Lie algebra.

The paper is organized as follows.  
In Section~2, we provide the required definitions and Rota--Baxter constructions.
In Section~3, we give the preLie counterexample.  
In Section~4, we prove the spectral necessary condition and give the postLie counterexample.

Throughout, $\Bbbk$ is a field of characteristic zero.  
In spectral arguments we extend scalars to an algebraic closure, which does not affect finite-dimensionality or injectivity.

The counterexamples were obtained with the assistance of ChatGPT 5.6 Sol.

\section{Preliminaries}

A \emph{preLie algebra} is a vector space $L$ with a bilinear product
$\tr$ satisfying
$$
(x\tr y)\tr z-x\tr(y\tr z)
  =(y\tr x)\tr z-y\tr(x\tr z).
$$
Equivalently, if $D_x(y)=x\tr y$, then $[D_x,D_y]=D_{x\tr y-y\tr x}$.

A \emph{postLie algebra} is a vector space $L$ with a Lie bracket $\{\,,\,\}$ and a bilinear product $\tr$ such that every $D_x\colon L\to L$, $D_x(y) = x\tr y$, is a derivation of the Lie algebra,
$D_x\{y,z\}=\{D_xy,z\}+\{y,D_xz\}$, and
$$
[D_x,D_y] = D_{[x,y]},\quad
[x,y] = \{x,y\} + x\tr y - y\tr x.
$$
The bracket $[\,,\,]$ is called the \emph{induced bracket}.
Thus $D\colon (L,[\,,\,])\to\Der(L,\{\,,\,\})$ is a Lie representation.
A~preLie algebra is precisely a postLie algebra whose initial bracket $\{\,,\,\}$ is zero.

A \emph{preassociative algebra} has two operations $\prec,\succ$ satisfying the following identities
$$
(a\prec b)\prec c=a\prec(b*c),\quad
(a\succ b)\prec c=a\succ(b\prec c),\quad
(a*b)\succ c=a\succ(b\succ c),
$$
where $a*b=a\succ b+a\prec b$.  
Then $*$ is associative and $a\tr b=a\succ b-b\prec a$ is preLie.

A \emph{postassociative algebra} has operations $\prec,\succ,\cdot$ such that $\cdot$ is associative and
the following six identities hold:
$$
\begin{gathered}
(a\prec b)\prec c=a\prec(b*c),\quad
(a\succ b)\prec c=a\succ(b\prec c),\\
(a*b)\succ c=a\succ(b\succ c),\quad
(a\succ b)\cdot c=a\succ(b\cdot c),\\
(a\prec b)\cdot c=a\cdot(b\succ c),\quad
(a\cdot b)\prec c=a\cdot(b\prec c).
\end{gathered}
$$
Here $a*b=a\succ b+a\prec b+a\cdot b$. 
The product $*$ is associative, and the operations
$\{a,b\}=a\cdot b-b\cdot a$ and
$a\tr b=a\succ b-b\prec a$ form a~postLie algebra.  
Its induced bracket is $[a,b]=a*b-b*a$.

\begin{definition}\label{def:ado-properties}
Let $L$ be finite-dimensional.
\begin{enumerate}
\item A preLie algebra $L$ has property $A_{\PreAs}$ if it embeds into
the induced preLie algebra of a finite-dimensional preassociative algebra.
\item A postLie algebra $L$ has property $A_{\PostAs}$ if it embeds into
the induced postLie algebra of a finite-dimensional postassociative algebra.
\end{enumerate}
A preLie algebra will also be regarded as a postLie algebra with zero initial bracket when property $A_{\PostAs}$ is considered.
\end{definition}

The PBW theorems in~\cite{DotsenkoTamaroff,DotsenkoPost,GubarevPBW} give universal enveloping algebras $U_{\PreAs}(L)$ and $U_{\PostAs}(L)$ and injective canonical maps from $L$.  
The following elementary reformulation makes the additional finite-dimensional requirement explicit.

\begin{proposition}\label{prop:finite-codim}
A preLie algebra $L$ has $A_{\PreAs}$ if and only if $U_{\PreAs}(L)$ has a preassociative ideal $I$ of finite codimension such that $I\cap L=0$.  
A postLie algebra $L$ has $A_{\PostAs}$ if and only if $U_{\PostAs}(L)$ has a postassociative ideal $I$ of finite codimension such that $I\cap L=0$.
\end{proposition}

Thus the PBW theorem proves that the zero ideal meets $L$ trivially, whereas an Ado
theorem would require a suitable cofinite ideal.

An associative Rota--Baxter algebra of weight $\theta$ is an
associative algebra $A$ with an operator $R$ satisfying
\begin{equation}\label{RB}
R(a)R(b)=R\bigl(R(a)b+aR(b)+\theta ab\bigr).
\end{equation}

We refer to~\cite{Guo} for standard constructions and terminology.

Recall that a Rota--Baxter operator~$R$ of weight~1 defined on an algebra~$A$
is called splitting if there are subalgebras $A_{\pm}$ of~$A$ such that
$A=A_-\oplus A_+$ is a vector-space direct sum of associative subalgebras and
$R(x_- + x_+)=-x_-$ for $x_\pm\in A_\pm$.

For $\theta=0$, the operations
$a\succ b=R(a)b$, $a\prec b=aR(b)$ are preassociative. 
For $\theta=1$, the same formulas together with $a\cdot b=ab$ are postassociative.  
For weight~1, we get a~postLie algebra with the initial bracket $\{u,v\} = uv - vu$ 
and 
\begin{equation}\label{eq:RB-post}
a\tr b = \{R(a),b\}.
\end{equation}

We need the converse only in the following explicit finite form.
The construction is the associative specialization of Theorems~3.8 and~3.11 in~\cite{GK}.

\begin{lemma}\label{lem:doubles}
a) Every finite-dimensional preassociative algebra embeds into a
finite-dimensional associative Rota--Baxter algebra of weight~0.

b) Every finite-dimensional postassociative algebra embeds into a
finite-dimensional associative Rota--Baxter algebra of weight~1 with a splitting Rota--Baxter operator.
\end{lemma}

The embeddings preserve the induced preLie and postLie operations.

Let $T$ be a preassociative or postassociative algebra, as appropriate.
Let $T'$ be a~second copy of $T$ and write $a'$ for the copy
of $a$.
In the preassociative case put $\widehat T=T\oplus T'$ and
$$
ab = a*b,\quad
ab' = (a\succ b)',\quad
a'b = (a\prec b)',\quad
a'b' = 0,
$$
where $a*b = a\succ b + a\prec b$, and define $R(a) = 0$, $R(a') = a$.
In the postassociative case use
$$
ab = a*b,\quad
ab' = (a\succ b)',\quad
a'b = (a\prec b)',\quad
a'b' = (a\cdot b)',
$$
where $a*b = a\succ b + a\prec b + a\cdot b$, and put $R(a) = -a$, $R(a') = a$.
Here $\widehat{T}_- = T$ and $\widehat{T}_+ = \{a+a'\mid a\in T\}$.
In both cases $a\to a'$ is the required embedding and preserves the induced preLie
or postLie structure, see~\cite{GK}.

\section{Failure for preLie algebras}

Recall that a Zinbiel algebra is a vector space with a product
$\diamond$ satisfying
\begin{equation}\label{eq:zinbiel}
(a\diamond b+b\diamond a)\diamond c
=a\diamond(b\diamond c).
\end{equation}
The symmetrization $a*b=a\diamond b+b\diamond a$ is commutative and
associative. 
Every finite-dimensional Zinbiel algebra is nilpotent by~\cite[Theorem~2.4]{Towers}.  
The convention in~\cite{Towers} is opposite to~\eqref{eq:zinbiel}; applying Theorem~2.4 to the opposite product gives the statement used here.  
Nilpotence means that all sufficiently long Zinbiel monomials vanish.  
This implies ordinary associative nilpotence of the symmetrized product: an $m$-fold $*$-product expands
as a sum of $\diamond$-monomials of length~$m$, and hence is zero for sufficiently large~$m$.  
The Zinbiel identity also permits one to reduce arbitrary bracketings to the standard one, so the stated
consequence agrees with the usual definitions of Zinbiel nilpotence.

\begin{lemma}\label{lem:one-generator-zinbiel}
Let $D$ be a preassociative algebra generated by one element $x$, and assume
that $x\succ x=x\prec x$.  Then $a\succ b=b\prec a$ for all $a,b\in D$.
Consequently, $D$ is a Zinbiel (precommutative) algebra with product
$a\diamond b=a\succ b$.
\end{lemma}

\begin{proof}
Let $F$~be the universal enveloping preassociative algebra of the one-dimensional preLie algebra with zero product. 
By~\cite[Example~4.4]{GubarevPBW}, $F$ is the free precommutative algebra on its generator.  
In particular, $a\succ b = b\prec a$ holds in~$F$.  
The assignment of its generator to~$x$ induces a surjective preassociative homomorphism $F\to D$, since $D$ is generated by~$x$ and the defining relation $x\succ x - x\prec x = 0$ holds in~$D$. 
Hence the same identity holds in~$D$.
Substituting $a\prec b = b\succ a$ into the third preassociative identity gives
$(a\diamond b+b\diamond a)\diamond c = a\diamond(b\diamond c)$,
so $\diamond$ is a~Zinbiel product.
\end{proof}

Recall that a preLie algebra $(L,\tr)$ is a~Novikov algebra if it also satisfies the identity $(x\tr y)\tr z = (x\tr z)\tr y$.

\begin{theorem}
\label{thm:preas-negative}
Let $\Bbbk$ be a field of characteristic zero and let $L=\Bbbk x\oplus\Bbbk y$ with
\begin{equation}\label{eq:preas-counter}
x\tr y = y, \quad
x\tr x = y\tr x = y\tr y = 0.
\end{equation}
Then $L$ is a preLie algebra which does not have $A_{\PreAs}$.
\end{theorem}

\begin{proof}
It is easy to check that $L$ is preLie.

Suppose that $L$ embeds into a finite-dimensional preassociative algebra
$D$, and denote the images again by $x,y$. 
Let $E$ be the preassociative subalgebra generated by $x$.  Since~$x\tr x=0$, one has
$x\succ x=x\prec x$.  
Lemma~\ref{lem:one-generator-zinbiel} shows directly that $E$ is a finite-dimensional Zinbiel algebra.  
By~\cite[Theorem~2.4]{Towers}, its symmetrized total associative product is nilpotent.  
In particular, for the total associative product $a*b=a\prec b+a\succ b$ in $D$, one has
$x^{*N}=0$ for some $N$.

Left and right multiplication by $x$ in the associative algebra $(D,*)$ commute.  
Since $x^{*N}=0$, both operators are nilpotent; hence $\ad_x=L_x-R_x$ is nilpotent.
On the other hand, the commutator of the total preassociative product is the commutator of the induced preLie product.
Indeed, $x*y - y*x = x\tr y - y\tr x = y$.  
Thus $\ad_x(y) = y$, contradicting nilpotence of $\ad_x$.

Finally, the right multiplication by $x$ in~\eqref{eq:preas-counter} is zero, while the right multiplication by~$y$ sends $x$ to $y$ and $y$ to zero.  
These right multiplications commute, so the counterexample is also a Novikov algebra.
\end{proof}

\begin{corollary}
Not every finite-dimensional preLie algebra embeds into a~finite-dimensional associative Rota--Baxter algebra of weight zero via $x\tr y=R(x)y-yR(x)$.
\end{corollary}

\begin{proof}
Such an embedding would induce a finite-dimensional preassociative embedding, contrary to Theorem~\ref{thm:preas-negative}.
\end{proof}

The two finite-dimensional enveloping-algebra properties for preLie algebras are quite different.

\begin{example}
\label{ex:separation}
The preLie algebra \eqref{eq:preas-counter} has $A_{\PostAs}$ although it does not have $A_{\PreAs}$.
\end{example}

\begin{proof}
Let $A = M_2(\Bbbk)$, let $A_+$ be the subalgebra of matrices whose bottom row is zero, and let $A_-$ be the subalgebra of matrices whose top row is zero. 
Then $A=A_+\oplus A_-$.  
Put $R=-\pr_{A_-}$; this is a splitting Rota--Baxter operator of weight~1. 
Take $a=-E$ and $b=e_{21}$. Then
$\{a,b\}=0$, $R(a)=e_{22}$, and $R(b)=-e_{21}$. Hence
$$
\{R(a),b\} = b,\quad 
\{R(a),a\} = 0,\quad
\{R(b),a\} = \{R(b),b\} = 0.
$$
The map $x\to a$, $y\to b$ is therefore an injective postLie homomorphism from $(L,\{\,,\,\}=0,\tr)$ into the postLie algebra induced by the finite-dimensional postassociative algebra associated with $(A,R)$.
\end{proof}

\section{Failure for postLie algebras}

Lemma~\ref{lem:resonance} provides the main technical part.
Its conclusion is specific to the finite-dimensional situation forced by Lemma~\ref{lem:doubles}.

\begin{lemma}\label{lem:resonance}
Let~$A$ be a~finite-dimensional associative algebra over a~field of characteristic zero, and let~$R$ be a~Rota--Baxter operator of weight~1 on~$A$. 
Write $\{u,v\}=uv-vu$ and $\ad_u(v)=\{u,v\}$. 
Suppose that~$a\in A$, $a\neq0$, satisfies $\{R(a),a\}=\lambda a$ with $\lambda\neq0$.  
Then~$a$ is nilpotent and
$\Spec(\ad_{R(a)})\subset\lambda\mathbb Z$.
\end{lemma}

\begin{proof}
Put~$t=R(a)$ and~$p=(R+\id)(a)=t+a$.

First,~$a$ is nilpotent. 
Indeed, for every~$m\ge1$ such that~$a^m\neq0$, one has~$\{t,a^m\} = m\lambda a^m$. 
Thus the nonzero powers~$a^m$ are eigenvectors of~$\ad_t$ with pairwise distinct eigenvalues~$m\lambda$. 
Since~$A$ is finite-dimensional, only finitely many of them can be nonzero, and hence~$a$ is nilpotent.

We will use the following fact:

\smallskip

($\star$) Let~$f(X)\in \Bbbk[X]$.
If $f(0) = 0$ and~$f(t)=f(p)$, then~$f(t)=0$.

\smallskip

Indeed, define the associative product~$x*y=xy+R(x)y+xR(y)$ on~$A$ and set~$S=R+\id$.
By~\eqref{RB}, $R(x*y)=R(x)R(y)$ and~$S(x*y)=S(x)S(y)$. Let~$a^{*1}=a$ and~$a^{*(n+1)}=a*a^{*n}$. Then, by induction,~$R(a^{*n})=t^n$ and~$S(a^{*n})=p^n$.
Let $f(X)=\sum_{n\ge1}c_nX^n$ and put~$z=\sum_{n\ge1}c_n a^{*n}$.
Then~$R(z)=f(t)$ and~$S(z)=f(p)$. If~$f(t)=f(p)$, then~$R(z)=S(z)$. 
Since~$S-R=\id$, we obtain~$z=0$, and therefore~$f(t)=R(z)=0$. 
Thus, if~$f(0)=0$ and~$f(t)=f(p)$, then~$f(t)=0$.

Since~$a$ is nilpotent, the derivation~$\ad_a$ is nilpotent. 
We have $\Phi := \exp(-\ad_a/\lambda)\in\Aut(A)$.
Since~$\{a,t\}=-\lambda a$ and~$\{a,\{a,t\}\}=0$, we have~$\Phi(t)=t+a=p$. 
Hence, for every polynomial~$f$ without constant term,~$f(p)=\Phi(f(t))$.
If~$f(t)$ commutes with~$a$, then~$\Phi(f(t))=f(t)$, and hence~$f(p)=f(t)$.

We now pass to the spectral argument. 
Choose a nonzero polynomial $M(X)$ of minimal degree such that
$M(0)=0$ and $M(t)=0$. After extending scalars to an algebraic closure,
write
$M(X)=X^r\prod_{\alpha\in\Sigma}(X-\alpha)^{n_\alpha}$, $r\geq1$,
where $\Sigma$ is the set of nonzero roots of~$M$.
All subsequent constructions in this proof are performed over
$\overline{\Bbbk}$; property~($\star$) remains valid after this scalar
extension.
For every~$\alpha\in\Sigma$, choose a~polynomial~$q_\alpha(X)$ such that
$$
q_\alpha(X)\equiv1\pmod{(X-\alpha)^{n_\alpha}}, \quad
q_\alpha(X)\equiv0\pmod{X^r\prod_{\beta\in\Sigma\setminus\{\alpha\}}(X-\beta)^{n_\beta}}.
$$
Choose each $q_\alpha$ of degree less than $\deg M$.  
Then $q_\alpha(0)=0$, and $e_\alpha=q_\alpha(t)$ is a~nonzero idempotent.  
The idempotents $e_\alpha$ are pairwise orthogonal.

Let~$L_t$ and~$R_t$ denote left and right multiplication by~$t$. They commute. 
For a~sufficiently large~$N$ put
$A_{\alpha,\beta} = \ker(L_t-\alpha\id)^N \cap \ker(R_t-\beta\id)^N$.
Then we have the joint generalized eigenspace decomposition
$$
A=\bigoplus_{\alpha,\beta\in\Sigma\cup\{0\}}A_{\alpha,\beta}.
$$
For~$\alpha,\beta\in\Sigma$,~$A_{\alpha,\beta}=e_\alpha A e_\beta$.

Define a~graph~$G$ on~$\Sigma$ by joining~$\alpha$ and~$\beta$ if~$e_\alpha a e_\beta\neq0$ or~$e_\beta a e_\alpha\neq0$.
If~$e_\alpha a e_\beta\neq0$, then this element is a~$\lambda$-eigenvector of~$\ad_t$, since
$$
\{t,e_\alpha a e_\beta\}
  =e_\alpha\{t,a\}e_\beta
  =\lambda e_\alpha a e_\beta.
$$
On the other hand, the only possible eigenvalue of
$\ad_t=L_t-R_t$ on~$A_{\alpha,\beta}$ is~$\alpha-\beta$.
Hence~$\alpha-\beta=\lambda$.
Similarly,~$e_\beta a e_\alpha\neq0$ gives~$\alpha-\beta=-\lambda$.

Let~$C$ be a~connected component of~$G$, and put~$e_C=\sum_{\alpha\in C}e_\alpha$.
Then~$e_C=q_C(t)$ for some $q_C\in \overline{\Bbbk}[X]$ with~$q_C(0)=0$.
If~$\{e_C,a\}=0$, then~$\Phi(e_C)=e_C$, and therefore~$q_C(p)=e_C=q_C(t)$.
By~($\star$), $e_C=0$, which is impossible. Hence~$\{e_C,a\}\neq0$.

Since~$C$ is a~connected component of~$G$, one has~$e_C a e_\beta=0$ and~$e_\beta a e_C=0$ for any~$\beta\in\Sigma\setminus C$.
Let us determine the support of~$\{e_C,a\}$ in the joint generalized
eigenspace decomposition.
For~$\alpha,\beta\in\Sigma$, one has~$A_{\alpha,\beta}= e_\alpha A e_\beta$, and the projection onto~$A_{\alpha,\beta}$ is given by~$x\to e_\alpha x e_\beta$. Therefore
$$
e_\alpha \{e_C,a\} e_\beta
 = e_\alpha e_C a e_\beta - e_\alpha a e_C e_\beta.
$$
If~$\alpha,\beta\in C$, then~$e_\alpha e_C=e_\alpha$ and~$e_C e_\beta=e_\beta$, so the right-hand side is~$e_\alpha a e_\beta-e_\alpha a e_\beta=0$.
If~$\alpha\in C$, $\beta\notin C$, then~$e_\alpha e_C=e_\alpha$ and~$e_C e_\beta=0$, so the right-hand side is~$e_\alpha a e_\beta=0$ by the definition of a~connected component.
The case~$\alpha\notin C$, $\beta\in C$ is similar, and if~$\alpha,\beta\notin C$, then both~$e_\alpha e_C$ and~$e_C e_\beta$ vanish. 
Hence
$e_\alpha \{e_C,a\} e_\beta=0$ for all~$\alpha,\beta\in\Sigma$.
Moreover, left multiplication by $e_C$ projects onto the direct sum of the left generalized eigenspaces indexed by $C$, and right multiplication by $e_C$ has the analogous property.  
Thus, the $A_{0,0}$-component also vanishes, and all nonzero components of~$\{e_C,a\}$ lie in
$\bigoplus_{\alpha\in C}(A_{\alpha,0}\oplus A_{0,\alpha})$.

On~$A_{\alpha,0}$ the only possible eigenvalue of~$\ad_t$ is~$\alpha$, and on~$A_{0,\alpha}$ it is~$-\alpha$. 
At the same time, the Jacobi identity in $(A,\{\,,\,\})$ and the hypothesis of the statement give
$$
\{t,\{e_C,a\}\}
  = \{e_C,\{t,a\}\}
  = \lambda\{e_C,a\}.
$$
Thus the nonzero vector~$\{e_C,a\}$ is a~$\lambda$-eigenvector of~$\ad_t$.
Since all its nonzero components lie in~$\bigoplus_{\alpha\in C}(A_{\alpha,0}\oplus A_{0,\alpha})$, we have $\alpha=\lambda$ or~$\alpha=-\lambda$ for some~$\alpha\in C$.

Every edge of the graph changes the spectral value by~$\lambda$ or~$-\lambda$. Since every connected component contains either~$\lambda$ or~$-\lambda$, it follows that~$\Sigma\subset\lambda\mathbb Z$.

Finally, the decomposition into the spaces~$A_{\alpha,\beta}$ is invariant under~$\ad_t$, and the only possible eigenvalue of~$\ad_t$ on $A_{\alpha,\beta}$ is~$\alpha-\beta$. 
Since~$0\in\lambda\mathbb Z$ and $\Sigma\subset\lambda\mathbb Z$, every eigenvalue of~$\ad_t$ on~$A$ belongs
to~$\lambda\mathbb Z$.
\end{proof}

\begin{theorem}\label{thm:post-negative}
Let $\Bbbk$ be a field of characteristic zero and let $L=\Bbbk x\oplus\Bbbk e$ have zero initial Lie bracket
$\{\,,\,\}$ and product
\begin{equation}\label{eq:half-example}
x\tr x = x,\quad
x\tr e = 1/2e,\quad
e\tr x = e\tr e = 0.
\end{equation}
Then $(L,\{\,,\,\},\tr)$ is a postLie algebra which does not have $A_{\PostAs}$.
\end{theorem}

\begin{proof}
It is easy to check that~$L$ is a postLie algebra.

Suppose that $L$ embeds into a finite-dimensional postassociative algebra~$M$.
By Lemma~\ref{lem:doubles}, we embed~$L$ further into a finite-dimensional Rota--Baxter associative algebra $(A,R)$ of weight~1 where $R$ is splitting.
Denote the images of $x,e$ by $a,v$.
The initial Lie bracket is zero, while the induced postLie product is given by~\eqref{eq:RB-post}; therefore
$$
\{R(a),a\} = a,\quad 
\{R(a),v\}  = 1/2v.
$$
Lemma~\ref{lem:resonance}, applied with $\lambda=1$, implies that every eigenvalue of $\ad_{R(a)}$ is an integer.  Since $v\neq0$ is an eigenvector with eigenvalue $1/2$, this is a contradiction.
\end{proof}

The failure persists within the class of Novikov algebras.

\begin{corollary}\label{cor:novikov}
The two-dimensional algebra
$$
x\tr x = x, \quad
x\tr e = 1/2e,\quad
e\tr x = e, \quad
e\tr e = 0
$$
is a Novikov algebra and does not have $A_{\PostAs}$.
\end{corollary}

\begin{proof}
It is known that this algebra is a~Novikov algebra~\cite{BaiMeng}.
The rest is the same as in the proof of Theorem~\ref{thm:post-negative}.
\end{proof}

More generally, the proof gives the following useful test.

\begin{corollary}\label{cor:eigen-test}
Let $L$ be a finite-dimensional postLie algebra with $A_{\PostAs}$.  
If nonzero elements $x,y\in L$ satisfy
$x\tr x = \lambda x$,
$x\tr y = \mu y$, $\lambda\neq0$,
then $\mu/\lambda\in\mathbb Z$.
\end{corollary}

\begin{proof}
Embed $L$ into a finite-dimensional postassociative algebra and then pass to the splitting Rota--Baxter double from Lemma~\ref{lem:doubles}. 
If $a,v$ are the images of $x,y$, respectively, then
$\{R(a),a\}=\lambda a$,
$\{R(a),v\}=\mu v$.
Lemma~\ref{lem:resonance} implies that $\mu\in\lambda\mathbb Z$, as required.
\end{proof}

\section*{Acknowledgements}

The author is grateful to Pavel Kolesnikov for helpful discussions.

The research was carried out within the framework of the Sobolev
Institute of Mathematics state contract (project FWNF-2026-0017).

\noindent Vsevolod Gubarev \\
Sobolev Institute of Mathematics \\
Acad. Koptyug ave. 4, 630090 Novosibirsk, Russia \\
Novosibirsk State University \\
Pirogova str. 1, 630090 Novosibirsk, Russia \\
e-mail: wsewolod89@gmail.com


\begin{thebibliography}{99}
\bibitem{Ado-0}
I. D. Ado,
Note on the representation of finite continuous groups by means of linear
substitutions, Izv. Fiz.-Mat. Obshch. (Kazan) {\bf 7} (1935), 1--43 (in Russian).

\bibitem{Ado}
I. D. Ado,
The representation of Lie algebras by matrices,
Uspekhi Mat. Nauk {\bf 2} (1947), no.~6(22), 159--173 (in Russian).

\bibitem{Aguiar}
M. Aguiar,
Pre-Poisson algebras,
Lett. Math. Phys. (4) {\bf 54} (2000), 263--277.

\bibitem{BaiMeng}
C. Bai and D. Meng,
The classification of Novikov algebras in low dimensions, J.~Phys. A: Math. Gen. (8) {\bf 34} (2001), 1581--1594.

\bibitem{BaiGuoNiPost}
C. Bai, L. Guo, and X. Ni,
Nonabelian generalized Lax pairs, the classical Yang--Baxter equation and PostLie algebras, Comm. Math. Phys. (2) {\bf 297} (2010), 553--596.

\bibitem{Burde}
D. Burde,
Left-symmetric algebras, or pre-Lie algebras in geometry and physics,
Cent. Eur. J. Math. (3) {\bf 4} (2006), 323--357.

\bibitem{BurdeDekimpeVercammen}
D. Burde, K. Dekimpe, and K. Vercammen,
Affine actions on Lie groups and post-Lie algebra structures,
Linear Algebra Appl. (5) {\bf 437} (2012), 1250--1263.

\bibitem{Chapoton}
F. Chapoton,
Un th\'eor\`eme de Cartier--Milnor--Moore--Quillen pour les big\`ebres
dendriformes et les alg\`ebres braces,
J. Pure Appl. Algebra (1) {\bf 168} (2002), 1--18.

\bibitem{ChenMo}
Y. Chen and Q. Mo,
Embedding dendriform algebra into its universal enveloping Rota--Baxter algebra,
Proc. Amer. Math. Soc. (12) {\bf 139} (2011), 4207--4216.

\bibitem{Dokas}
I. Dokas,
Pre-Lie algebras in positive characteristic,
J. Lie Theory (4) {\bf 23} (2013), 937--952.

\bibitem{DotsenkoPost}
V. Dotsenko,
Functorial PBW theorems for post-Lie algebras, Comm. Algebra (5) {\bf 48} (2020), 2072--2080.

\bibitem{DotsenkoTamaroff}
V. Dotsenko and P. Tamaroff,
Endofunctors and Poincar\'e--Birkhoff--Witt theorems,
Int. Math. Res. Not. (16) {\bf 2021} (2021), 12670--12690.

\bibitem{EbrahimiFard}
K. Ebrahimi-Fard,
Loday-type algebras and the Rota--Baxter relation,
Lett. Math. Phys. (2) {\bf 61} (2002), 139--147.

\bibitem{EbrahimiFardGuo}
K. Ebrahimi-Fard and L. Guo,
Rota--Baxter algebras and dendriform algebras,
J. Pure Appl. Algebra (2) {\bf 212} (2008), 320--339.

\bibitem{EFLM}
K.~Ebrahimi-Fard, A.~Lundervold, and H.~Z.~Munthe-Kaas,
On the Lie enveloping algebra of a post-Lie algebra,
J. Lie Theory (4) {\bf 25} (2015), 1139--1165.

\bibitem{Gerstenhaber}
M. Gerstenhaber,
The cohomology structure of an associative ring,
Ann. Math. (2) {\bf 78} (1963), 267--288.

\bibitem{GubarevRBAs}
V. Gubarev,
Universal enveloping associative Rota--Baxter algebras of preassociative and postassociative algebra, 
J. Algebra {\bf 516} (2018), 298--328.

\bibitem{GubarevPreLie}
V. Gubarev,
Embedding of pre-Lie algebras into preassociative algebras,
Algebra Colloq. (2) {\bf 27} (2020), 299--310.

\bibitem{GubarevPostLie}
V. Gubarev,
Embedding of post-Lie algebras into postassociative algebras,
in {\it New Trends in Algebras and Combinatorics}, World Scientific,
Singapore, 2020, 57--67.

\bibitem{GubarevPBW}
V. Gubarev,
Poincar\'e--Birkhoff--Witt theorem for pre-Lie and post-Lie algebras,
J.~Lie Theory (1) {\bf 30} (2020), 223--238.

\bibitem{GK}
V. Gubarev and P. Kolesnikov,
Embedding of dendriform algebras into Rota--Baxter algebras,
Cent. Eur. J. Math. (2) {\bf 11} (2013), 226--245.

\bibitem{Guo}
L. Guo,
{\it An Introduction to Rota--Baxter Algebra},
Surveys of Modern Mathematics, vol.~4, International Press, Somerville,
and Higher Education Press, Beijing, 2012.

\bibitem{KolesnikovLeibniz}
P.S. Kolesnikov,
Conformal representations of Leibniz algebras,
Sib. Math. J. (3) {\bf 49} (2008), 429--435.

\bibitem{KolesnikovAdo}
P. Kolesnikov,
The Ado theorem for finite Lie conformal algebras with Levi decomposition,
J. Algebra Appl. (7) {\bf 15} (2016), 1650130.

\bibitem{KolesnikovPreAs}
P. Kolesnikov,
Gr\"obner--Shirshov bases for pre-associative algebras,
Comm. Algebra (12) {\bf 45} (2017), 5283--5296.

\bibitem{LodayDialgebras}
J.-L. Loday,
Dialgebras,
in {\it Dialgebras and Related Operads}, Lecture Notes in Math. 1763,
Springer, Berlin, 2001, 7--66.

\bibitem{LodayRonco}
J.-L. Loday and M. Ronco,
Trialgebras and families of polytopes,
in {\it Homotopy Theory: Relations with Algebraic Geometry, Group Cohomology,
and Algebraic $K$-Theory}, Contemp. Math. {\bf 346}, Amer. Math. Soc., Providence, 2004,
369--398.

\bibitem{Manchon}
D. Manchon,
A short survey on pre-Lie algebras,
in {\it Noncommutative Geometry and Physics: Renormalisation, Motives, Index
Theory}, EMS, Z\"urich, 2011, 89--102.

\bibitem{MikhalevShestakov}
A. A. Mikhalev and I. P. Shestakov,
PBW-pairs of varieties of linear algebras,
Comm. Algebra (2) {\bf 42} (2014), 667--687.

\bibitem{MostovoyPerezShestakov}
J. Mostovoy, J. M. P\'erez-Izquierdo, and I. P. Shestakov,
Hopf algebras in non-associative Lie theory,
Bull. Math. Sci. (1) {\bf 4} (2014), 129--173.

\bibitem{Ronco}
M. Ronco,
Eulerian idempotents and Milnor--Moore theorem for certain noncocommutative
Hopf algebras, J. Algebra {\bf 254} (2002), 152--172.

\bibitem{Towers}
D. A. Towers,
Zinbiel algebras are nilpotent, J. Algebra Appl. (8) {\bf 22} (2023), 2350166.

\bibitem{Vallette}
B. Vallette,
Homology of generalized partition posets, J. Pure Appl. Algebra (2) {\bf 208} (2007), 699--725.

\bibitem{Vinberg}
E. B. Vinberg,
The theory of homogeneous convex cones, Tr. Mosk. Mat. Obs. {\bf 12} (1963), 303--358.
\end{thebibliography}
\end{document}